\documentclass[12pt]{amsart}
\usepackage{amsmath, amsthm, amscd, amsfonts, amssymb, graphicx, color}
\usepackage[bookmarksnumbered, colorlinks, plainpages]{hyperref}
\hypersetup{colorlinks=true,linkcolor=red, anchorcolor=green, citecolor=cyan, urlcolor=red, filecolor=magenta, pdftoolbar=true}
\usepackage{amsmath}
\usepackage{amsthm}
\usepackage{amssymb}
\usepackage[latin9]{inputenc}
\DeclareMathOperator\supp{supp}
\newtheorem{theorem}{Theorem}[section]
\newtheorem{lemma}[theorem]{Lemma}
\newtheorem{prop}[theorem]{Proposition}
\newtheorem{corollary}[theorem]{Corollary}
\theoremstyle{definition}

\theoremstyle{remark}
\newtheorem{remark}[theorem]{Remark}
\numberwithin{equation}{section}
\newcommand{\la}{\lambda}

\newcommand{\F}{{\Bbb F}}
\newcommand{\cal}{\mathcal}

\newcommand{\eps}{\varepsilon}

\begin{document}
\setcounter{page}{1}
	
\title[paley inequality]
{Paley's inequality for nonabelian groups  }

\author[C. Y. Chuah]{C. Y. Chuah}

\address{Chian Yeong Chuah, Department of Mathematics, Baylor University, 1301 S University Parks Dr, Waco, TX 76798, USA.}
\email{\textcolor[rgb]{0.00,0.00,0.84}{Chian\_Chuah@baylor.edu}}
	
\author[Y. Han]{Y. Han}
	
\address{Yazhou Han, College of Mathematics and Systems Science, Xinjiang
		University, Urumqi 830046, China}
\email{\textcolor[rgb]{0.00,0.00,0.84}{hyz0080@aliyun.com}}

\author[Z. Liu]{Z. Liu}
\address{Zhenchuan Liu,
Department of Mathematics
Baylor University
1301 S University Parks Dr, Waco, TX 76798, USA.}
\email{\textcolor[rgb]{0.00,0.00,0.84}{Zhen-chuan\_Liu1@baylor.edu}}

\author[T. Mei] {T. Mei }  \address{Tao Mei, Department of Mathematics
Baylor University
1301 S University Parks Dr, Waco, TX 76798, USA.}
\email{\textcolor[rgb]{0.00,0.00,0.84}{tao\_mei@baylor.edu}}


\maketitle
  \begin{abstract} This article studies  Paley's theory for lacunary Fourier series on (nonabelian) discrete groups. 
  The results  unify and generalize  the work of Rudin (\cite{Ru62} Section 8) for abelian discrete groups and the work of   Lust-Piquard and Pisier (\cite{LP91}) for operator valued functions, and provide new examples of Paley sequences and $\Lambda(p)$ sets on free groups.
  
\end{abstract}
 
\section*{Introduction}

Denote by ${\Bbb T}$ the unit circle. Given a lacunary sequence $  (j_k)_{ k\in{\Bbb N}}\in {\Bbb Z}$, i.e. $$\frac {|j_{k+1}|}{j_k}>1+\delta$$ for some $\delta>0$, the classical Khintchine's inequality says that 
\begin{eqnarray*}
(\sum_k |c_k|^2)^\frac12\simeq ^{c_\delta} \|\sum_k c_kz^{j_k}\|_{L^1({\Bbb T})} .
\end{eqnarray*}
This shows that $\ell_2$ embeds into $L^1$. However, the projection $$P: f \mapsto \hat f(j_k)$$ is NOT bounded from $L^1({\Bbb T})$ to $\ell_2$. Here $\hat f$ denotes for the Fourier transform of $f$. 
This can be easily seen by looking at the so-called Riesz products.
 Paley's theory is an improvement of  Khintchine's inequality. It says that,
$$(\sum_k |c_k|^2)^\frac12\simeq ^{c_\delta}\inf \{\|f\|_{L^1}; f\in L^1({\Bbb T}), {\rm supp} \hat f\subset {\Bbb N}, \hat f(j_k)=c_k\}.$$
This shows that  the projection $P$ is bounded from the analytic $L^1$ to $\ell_2$, which has important applications, e.g. to Grothendieck's theory on 1-summing maps.

 Let $H^1({\Bbb T})$ be the real Hardy space on the unite circle, that consists of integrable functions  such that both their analytic and the anti-analytic parts   are integrable. 
 Paley's theory says that 
  \begin{eqnarray}
  (\sum_k |c_k|^2)^\frac12\simeq ^{c_\delta}\inf \{\|f\|_{H^1}; f\in H^1({\Bbb T}), \hat f(j_k)=c_k,\forall j_k\in E\},\label{paley1}
  \end{eqnarray}
  for $E=\{j_k, k\in{\Bbb N}\}\subset{\Bbb Z}$.
Let us call $E\subset {\Bbb N}$ a Paley set if the above equivalence holds for all $(c_k)_k\in \ell_2$.   Rudin proved that $E$ is a Paley set only if $$\sup_{n\in{\Bbb N}} \# E\cap [2^n,2^{n+1}]<C$$ which is equivalent to say that $E$ is a finite union of lacunary sequences. 
  
  By Fefferman-Stein's  $H^1$-BMO duality theory, (\ref{paley1}) has an equivalent formulation that, for any $c_k\in \ell_2$,
 \begin{eqnarray}\label{paley2}
 (\sum_k |c_k|^2)^\frac12\simeq ^{c_\delta}\|\sum_k c_kz^{j_k}\|_{BMO({\Bbb T})}.
  \end{eqnarray}
  Here BMO denotes the bounded mean oscillation  (semi)norm
    $$\|g\|_{BMO}=\sup _I \frac1{|I|}\int_I |g-g_I|$$ with the supremum taking for all arc $I\in {\Bbb T}$.

 In the first part of this article, we give   an interpretation of Paley's theory in the semigroup language which allows an extention to non-abelian discrete groups. 
 Let $P_t, t>0,$ to be the Poisson integral operator that sends $e^{ik\theta}$ to $r^{|k|}e^{ik\theta}$ with $r=e^{-t}$. Here is an equivalent characterization of the  classical BMO and $H^1$-norms by $P_t$'s. That, for $f\in L^1({\Bbb T}$),
 \begin{eqnarray*}
 \|f\|_{BMO}\simeq \sup_t\|(P_t|f-P_tf|^2)\|_{L^\infty({\Bbb T})}^\frac12\\
 \|f\|_{H^1}\simeq \|(\int_0^\infty|\partial P_tf|^2tdt)^\frac12\|_{L^1({\Bbb T})}^\frac12.
 \end{eqnarray*}

Consider a discrete group $G$ and a conditionally negative   length  $\psi$ on $G$.  By that, we mean  $\psi$ is a ${\Bbb R}_+$-valued function on $G$ satisfying $\psi(g)=0$ iff $g=e$, $\psi(g)=\psi(g^{-1})$,  and
\begin{eqnarray}\label{CN}
  \sum_{g,h}\overline{a_g}a_h\psi(g^{-1}h)\leq0
\end{eqnarray}
for any finite collection of coefficients $a_g\in {\Bbb C}$ with $\sum_g a_g=0$. 
We say a sequence $h_k\in G$ is $\psi$-lacunary if there exists a constant $\delta>0$ such that
\begin{eqnarray}
   \psi(h_{k+1}) &\geq& (1 + \delta){\psi(h_k)} \label{psi1} \\
 \psi(h_k^{-1}h_{k'})&\geq& \delta\max\{\psi(h_k),\psi(h_{k'})\}.\label{psi2}
 \end{eqnarray} 
 for any $k,k'$.   Note the second condition follows from the first one if we require $\psi$ is sub-additive, i.e. $\psi(hg)\leq C\psi(h)+\psi(g)$. Let $\la$ be the regular left representation of $G$. We say $$x=\sum_k c_k\la_{h_k}$$ is a $\psi$-lacunary operator valued-Fourier series if the sequence $h_k$ is $\psi$-lacunary. We allow $c_k$ to take value in $B(H)$ (or another von Neumann algebra), so $c_k\la_{h_k}$ is understood as $c_k\otimes\la_{h_k}$ in that case. We say $ x $ is a  lacunary Fourier series if there is a conditionally negative $\psi$ so that  $h_k$ is $\psi$-lacunary. 

  
   Let $$T_t:\la_g\mapsto e^{-t\psi(g)}\la_g$$ be the semigroup associated with $\psi$. We will show that,

   \begin {theorem} \label{theorem1}Assume $(h_k)$ is a $\psi$-lacunary sequence. Then, for any sequence $c_k\in B(H)$,
 \begin{eqnarray}
 \|\sum_k c_k\la_{h_k}\|_{BMO_c(\psi)}^2&\simeq^{ c_\delta }&\|\sum_k |c_k|^2 \|. \label{BMOc1}\\
\inf \{\|x\|_{H^1_c(\psi)}; \hat x(h_k)=c_k,\forall k\in{\Bbb N}\}&\simeq^{ c_\delta}& tr (\sum_k |c_k|^2)^\frac12. \label{H1}
 \end{eqnarray}
  \end{theorem}
\noindent  Here the semigroup-$H^1$ and BMO-norms are defined as
  \begin{eqnarray*} 
  \|x\|_{H^1_c(\psi)}=tr\otimes\tau (\int_0^\infty |\frac{\partial_s T_sx}{\partial s}|^2sds)^\frac12
 \\
   \|x\|_{{\rm BMO}_c(\psi)}=\sup_s \|T_s|x-T_sx|^2\|^\frac12.
  \end{eqnarray*}
  with $tr, \tau$   the canonical traces on $B(H)$ and the reduced $C^*$ algebra of $G$. By taking adjoints, one gets the estimate on the full BMO spaces and obtain that every set of a $\psi$-lacunary sequence is a $\Lambda_p$ set for  all $2<p<\infty$ by interpolation (\cite{JM12}). More precisely, we have that, for any $p>2, x=\sum_kc_k\otimes\la _  {h_k}$ with $ c_k$ belonging to the Schatten-p classes, 
\begin{eqnarray}
\ \  \|x\|_{L^p(\hat G, S^p)}\leq c^{\frac {p-2}{2p}}_\delta p\max\{\|(\sum_k |c_k|^2)^\frac12 \|_{S^ p},\|(\sum_k |c_k^*|^2)^\frac12 \|_{S^ p}\}.
 \end{eqnarray}

  (\ref{BMOc1}) and its adjoint version  also imply that the row and column semigroup BMO norms differ from each other with constants at least $\simeq\sqrt n$  for $n$ by $n$ matrix-valued functions (see Remark \ref{finalremark}).  
  
  Lust-Piquard and Pisier's work  (\cite{LP91}) is the first in the study of noncommutative  Paley's inequality. They overcomed the difficulties due to the noncommutativity by combining the row and column spaces. An interesting point of  Theorem \ref{theorem1} is that it also gives  interpretations of the row (and column) version of noncommutative Paley's inequality separately. 
  
 In the second part of the article, we assume the group $G$ is equipped with  a bi-invariant order ``$\leq$". Let $G_+=\{g\in G;  e\leq g\}$. Following Rudin's terminology \cite{Ru62}, we say a subset $E\in G_+$ is lacunary if  there exists a constant $K$ such that
$$N(E)=\sup_{g\in G_+}\#\{h\in E: g\leq h\leq g^2\}\leq K.$$ 
 
\begin{theorem}\label{theorem2}
For any sequence $c_k\in S^1(H)$, and
any sequence $\{g_k\}_{k=1}^\infty$ in a lacunary subset $E\subset G_+$, we have
 \begin{eqnarray}
\inf \{tr\otimes\tau |x|; \hat x(g_k)=c_k, \supp\hat x\subset G_+\}&\simeq  \|(c_k)\|_{S^1(\ell^2_{cr})}, 
 \end{eqnarray}
\end{theorem}

 Theorem \ref{theorem2} follows from a factorization theorem of noncommutative analytic Hardy spaces and   an adaption of   Lust-Piquard and Pisier's argument (\cite{LP91}) to Rudin's terminology of lacunary sets. The authors feel   happy that it works out and provides  interesting examples of Paley sequences and $\Lambda(p)$ sets (see e.g. Corollary \ref{corlast}) on free groups.

  \section{Noncommutative $L^p$- space}
Let $\mathcal{M}$ be a
semifinite von Neumann algebra acting on a seperable Hilbert space $\mathcal{H}$ with
a normal semifinite faithful trace
$\tau$. 
For $0<p<\infty$, denote by $L^p(\mathcal{M})$ the noncommutative $L^p$ spaces associated with the (quasi)norm
$\|x\|_p=\tau(|x|^p)^{\frac{1}{p}}$.  As usual, we set $L^\infty( \mathcal{M})=\mathcal{M}$
equipped with the operator norm.  
 For a (nonabelian) discrete group $G$, the von Neumann algebra is the closure of the linear span of left regular representation $\la_g$'s w.r.t. a weak operator topology. The trace $\tau$ is simply defined as 
 $$\tau x=c_e,$$
 for $x=\sum_g c_g\la_g$.
 The associated $L_p$ norm is defined as $$\|x\|_p=(\tau |x|^p)^\frac1p$$ for   $1\leq p<\infty$.  
 When $G$ is abeliean, e.g. $G={\mathbb Z}^d$, the obtained $L_p$ space is the $L_p$ space on the dual group e.g. $\widehat{{\mathbb Z}^d}={\mathbb T}^d$. 
We will denote by ${\cal G}$ the group von Neumann algebra of $G$, and by $L^p(\hat G)$ the associated noncommutative $L^p$ spaces.
Let $\mathcal{M}=\mathcal{B}(\mathcal{H})$, the algebra of all
bounded operators on $\mathcal{H}$,
and $\tau=tr$, the usual trace on
$\mathcal{B}(\mathcal{H})$. Then the associated $L^p$-space
$L^p(\mathcal{M})$ is the Schatten
class $S^p(\mathcal{H})$. We refer the readers to the survey paper \cite{PX03} for more information on noncommutative $L_p$ spaces.

\subsection{Column and row spaces}

Let $0<p\leq\infty$	and let  $(x_{n})_{n\geq0}$ be a finite sequence in $L^p(\mathcal{M})$.
Define
\[
\|(x_{n})\|_{L^p(\mathcal{M}, \ell_c^2)}=\|(\sum\limits_{n\geq0}
|x_{n}|^{2})^{1/2}\|_p,~~
\|(x_{n})\|_{L^p(\mathcal{M}, \ell_r^2)}=\|(\sum\limits_{n\geq 0}
|x_{n}^{*}|^{2})^{1/2}\|_p.
\]
For $0<p<\infty,$ we define $L^p(\mathcal{M}, \ell^2_c)$(resp. $L^p(\mathcal{M}, \ell^2_r)$)
as the completion of
the family of all finite sequences in $L^p(\mathcal{M})$ with respect to
$\|\cdot\|_{L^p(\mathcal{M}, \ell_c^2)}$(resp. $\|\cdot\|_{L^p(\mathcal{M}, \ell_r^2)}$).
For $p=\infty$, we define $L^\infty(\mathcal{M}, \ell^2_c)$(resp. $L^\infty(\mathcal{M}, \ell^2_r)$)
as the Banach space of (possible infinite) sequences in $\mathcal{M}$ such that
$\sum_{n}x_{n}^*x_n$ (resp. $\sum_{n}x_{n}^*x_n$) converges in the
$w^*$-topology.

Let $0< p\leq\infty.$ We define the space $L^p(\mathcal{M}, \ell_{rc}^2)$ as follows:
 \begin{enumerate}
 \item  If $0<p<2$,
\[L^p(\mathcal{M}, \ell_{rc}^2)=L^p(\mathcal{M}, \ell_c^2)+L^p(\mathcal{M}, \ell_r^2)\]	
equipped with the norm:
\[
\|(x_k)_{n\geq0}\|_{L^p(\mathcal{M}, \ell_{rc}^2)}=\inf_{x_k=x_k^\prime+x_k^{\prime\prime}}
\{\|(x_k^{\prime})\|_{L^p(\mathcal{M}, \ell_c^2)}
+\|(x_k^{\prime\prime})\|_{L^p(\mathcal{M}, \ell_r^2)}\}
\]
where the infimum is taken over all decompositions for which
\[\|(x_k^{\prime})\|_{L^p(\mathcal{M}, \ell_c^2)}<\infty~\mbox{and}~
\|(x_k^{\prime\prime})\|_{L^p(\mathcal{M}, \ell_r^2)}<\infty.\]
\item  If $p\geq2$,
\[L^p(\mathcal{M}, \ell_{rc}^2)=L^p(\mathcal{M}, \ell_c^2)\cap L^p(\mathcal{M}, \ell_r^2)\]	
equipped with the  norm:
\[
\|(x_k)\|_{L^p(\mathcal{M}, \ell_{rc}^2)}=\max
\{\|(x_k)\|_{L^p(\mathcal{M}, \ell_r^2)}, \|(x_k)\|_{L^p(\mathcal{M}, \ell_c^2)}\}.
\]
 \end{enumerate}
 We will denote simply
by $S^p(\ell_c^2)$, $S^p(\ell_r^2)$ and $S^p(\ell_{rc}^2)$ the spaces
$L^p(\mathcal{M}, \ell_c^2)$, $L^p(\mathcal{M}, \ell_r^2)$ and $L^p(\mathcal{M}, \ell_{rc}^2)$
when $\mathcal{M}=\mathcal{B}(\mathcal{H})$, respectively.
Please refer to \cite{PX03, JMX06} for details on these spaces.

We will denote by $K(H)$ the collection of all compact operators on a Hilbert space $H$. The expression $X\lesssim Y$ means that there exists a positive constant $C$ such that
$X\leq CY $. To specify the dependence of this constant on additional parameters e.g. on
$p$ we write $X \lesssim_p Y$. If $X\lesssim Y$ and $Y\lesssim X,$ we write $X\thickapprox Y$.	
	
\section{Proof of Theorem \ref{theorem1}-- the BMO estimate.}

Given a conditionally negative length   $\psi$ on $G$,
Schoenberg's
theorem says that $$T_t: \lambda_g=e^{-t\psi(g)}\lambda_g$$ extends to a symmetric Markov semigroup of operators on the group von Neumann algebra $L^p(\hat G), 1\leq p\leq \infty$.
Following  \cite{JM12} and \cite{M08}, let us set
\begin{eqnarray}
\|x\|_{\mathrm{BMO}_c(\psi ) }&=&\sup_{0<t<\infty}\| T_{t }|x-T_{t}x|^2\| ^\frac12, \label{BMOT}
\end{eqnarray}
for $x\in L^2(\hat G)$. Let ${\mathrm{BMO}(\psi ) }$ be the space of all $ x\in L^2(\hat G)$ such that
\begin{eqnarray}
  \|x\|_{BMO(\psi)}=\max\{\|x\|_{ BMO_c(\psi)},\|x^*\|_{BMO_c(\psi)}\}<\infty.
\end{eqnarray}

When $G$ is the integer group ${\Bbb Z}$, $\hat G$ is the unit torus ${\Bbb T}$. The semigroup $T_t$ is the heat semigroup (resp. Poisson semigroup) if we set   $\psi(g)=|g|^2$ (resp. $|g|$) for $g\in {\Bbb Z}$. It is an elementary calculation that the semigroup BMO norm defined above coincides with the classical one. In the case of operator valued $x$, that is $x=\sum c_h\la_h$ with $c_h$ taking values in another von Neumann algebra, the semigroup BMO norm defined above coincides with the ones studied in \cite{M07}. The semigroup BMO norms may differ from each other for different semigroups, see \cite{FMS19} Section 4 for examples.

 \begin{lemma} \label {JM12}([JM12]) We have the following interpolation result 
$$ [BMO(\psi),L^1(\hat G)]_{\frac{1}{p}} = L^p(\hat G) $$
for
 $1<p<\infty$.     
\end{lemma}

\begin{lemma}\label{CS}
For $a_s\in {\Bbb R}_+$, $c_s,b_s\in B(H)$, we have, for any $0<p,q,r<\infty, \frac1p+\frac1q=\frac1r$,
\begin{eqnarray}
\|\sum_s a_sc^*_sb_s\|&\leq &\|\sum_s |c_s|^2 a_s\|^\frac12\|\sum_s |b_s|^2 a_s\|^\frac12\\
\|\sum_k a_sc^*_sb_s\|_{S^r }&\leq& \|\sum_s |c_s|^2 a_s\|_{S^p}^\frac12\|\sum_s |b_s|^2 a_s\|_{S^q}^\frac12.
\end{eqnarray}
\end{lemma}
\begin{proof} This is simply the Cauchy-Schwartz inequality.
\end{proof}
   \begin{lemma}\label{Lemma}
   Let $x=\sum_k c_k\la_{h_k}\in L^2(\hat G)$, Then, we have
 \begin{eqnarray}
 \tau( \int_0^\infty|\partial_s T_sx|^2sds)^\frac12&\leq &\frac12 (\sum_k |c_k|^2 )^\frac12. \label{H1<}
 \end{eqnarray}
Moreover, if we assume $(h_k)$ is a $\psi$-lacunary sequence, then
\begin{eqnarray}
  \| \int_0^\infty|\partial_s T_sx|^2sds\|&\leq &c_\delta\|\sum_k |c_k|^2  \|.
 \end{eqnarray}
 \end{lemma}

\begin{proof} An elementary calculation shows that
\begin{eqnarray*}
 \int_0^\infty|\partial_s T_sx|^2sds&=&\sum _{k,j}   (c_k\la_{h_k})^*c_j\la_{h_j}\psi(h_i)\psi(h_j)\int_0^\infty e^{-s(\psi(h_k)+\psi(h_j))}sds\\
  &=& \sum _{k,j}  a_{k,j}(c_k\la_{h_k})^*c_j\la_{h_j},
 \end{eqnarray*}
with $$a_{k,j}= \frac {\psi(h_k)\psi(h_j)}{(\psi(h_k)+\psi(h_j))^2}\geq0$$ since $ \int_0^\infty e^{-t}tdt=1$.
So
\begin{eqnarray*}
   \tau(\int_0^\infty|\partial_s T_sx|^2sds)^\frac12 &\leq& (\tau\int_0^\infty|\partial_s T_sx|^2sds)^\frac12\\
  &= &(\sum_{k}|c_k|^  2a_{k,k})^\frac12=\frac12  (\sum_k |c_k|^2)^\frac12.
   \end{eqnarray*}
On the other hand, it is easy to see that 
\begin{eqnarray*}
  \sup_j \sum_k a_{k,j}\leq c_\delta, \ \ \sup_k \sum_j a_{k,j}\leq c_\delta.
\end{eqnarray*}
 Applying Lemma 2 for $p=q=\infty$, we have \begin{eqnarray*}
   \|\int_0^\infty|\partial_s T_sx|^2sds\|
  &\leq &\|(\sum _{k,j} |c_k|^2a_{k,j})^\frac12\| \|(\sum_{k,j}|c_j|^  2a_{k,j})^\frac12\|\\
  &\leq& c_\delta  \|\sum_k |c_k|^2\|.
 \end{eqnarray*}

 \end{proof}
  
 \begin{theorem}\label{main1}
 Assume $(h_k)$ is a $\psi$-lacunary sequence. Then, for any $x=\sum_kc_k\la _  {h_k} $, we have
 \begin{eqnarray}
 \|x\|_{BMO_c(\psi)}^2&\simeq^{ c_\delta}& \|\sum_k |c_k|^2 \|.\label{BMOce}\\
 \|x\|_{BMO(\psi)}^2&\simeq^{ c_\delta}&\max\{\|\sum_k |c_k|^2 \|,\|\sum_k |c_k^*|^2 \|\}.\\
 \ \ tr (\sum_k |c_k|^2 )^\frac12&\simeq^{ c_\delta}&\inf\{\tau\otimes tr( \int_0^\infty|\partial_s T_sx|^2s)^\frac12 ; \tau (x \la^*_{ h_k})=c_k\}.\label{H1e}
 \end{eqnarray}
 \end{theorem}
\begin{proof}We prove the BMO estimate first. An easy calculation shows that
\begin{eqnarray*}
 T_t|x -T_tx|^2
  = \sum _{k,j}  a_{k,j}(c_k\la_{h_k})^*c_j\la_{h_j},
 \end{eqnarray*}
with $$a_{k,j}= e^{-t\psi(h^{-1}_kh_j )}(1-e^{-t\psi(h_k^{-1})})(1-e^{-t\psi(h_j)})\geq0.$$
By the lacunary property $\psi(h_k^{-1}h_j)\geq |\psi(h_k)-\psi(h_j)|$, we have
\begin{eqnarray*}
  \sum _{k}  a_{k,j}
  &\leq &\sum _{t\psi(h_k)\leq 1} (1-e^{-t\psi(h_k^{-1})})    +\sum _{t\psi(h_k)> 1}   e^{-t\psi(h^{-1}_kh_j )} \\ 
  &\leq &\sum _{t\psi(h_k)\leq 1}  t  \psi(h_k)   +\sum _{t\psi(h_k)> 1}    e^{- t \delta\psi(h_k)}  \\
  &\leq & 1+\delta^{-1}+\frac1{1-e^{-  {\delta^2} }}\\
   &\leq& c _\delta.  
 \end{eqnarray*}
We then get $\sup_j \sum_k a_{k,j}\leq c_\delta$. Similarly, $\sup_k \sum_j a_{k,j}\leq c_\delta.$  By Lemma \ref{CS}, we have
\begin{eqnarray*}
 \|T_t|x-T_tx|^2\|
  &\leq &\|\sum _{k,j} |c_k|^2a_{k,j}\|^\frac12 \|\sum_{k,j}|c_j|^  2a_{k,j}\|^\frac12\\
    &\leq&c_\delta  \|\sum_k |c_k|^2 \|.
 \end{eqnarray*}
 Taking supremum on $t$, we get $\|x\|^2_{BMO_c }\leq c_\delta \|\sum_k |c_k|^2 \|.$ The lower estimate is obvious by taking conditional expectation $\tau$ and sending $t$ to $\infty$. Taking the adjoint, we prove the  estimate for the BMO norms. 
 
 We now turn to the $H^1$-estimate (\ref{H1e}). By duality, we may choose $b_k$ such that $\|\sum |b_k|^2\|=1$ and
\begin{eqnarray*}tr (\sum_k |c_k|^2 )^\frac12= tr\sum c^*_kb_k=(\tau\otimes tr) x^*y,
  \end{eqnarray*} 
with $y=\sum b_k\la_{h_k}$ and any (finite) Fourier sum $x$ such that $\tau (x\la^*_{h_k})=c_k$.
We then have
\begin{eqnarray*}
(\tau\otimes tr)x^*y&=&4\tau\otimes tr \int_0^\infty \partial_s T_sx^*\partial_s T_sy sds\\
 &\leq& 4\tau\otimes tr (\int_0^\infty|\partial_s T_sx|^2sds)^\frac12\|\int_0^\infty|\partial_s T_sy|^2sds\|^\frac12
\end{eqnarray*}
 Combining the above estimates with Lemma \ref{Lemma}, we obtain \begin{eqnarray*}
 tr (\sum_k |c_k|^2 )^\frac12\leq 4c_\delta \tau\otimes tr(\int_0^\infty|\partial_s T_sx|^2sds)^\frac12.
    \end{eqnarray*}
    The other direction follows by taking $tr$ on both sides of (\ref{H1<}).

 \end{proof}

Given a length-lacunary sequence $h_k\in G$,  define the linear map $T$ from $L^\infty(\ell_2)$ to $BMO$ by 
 $$T((c_k))=\sum_k c_k\la_{h_k}.$$
 Then $T$ has a norm $c_\delta$ from $L^\infty(\ell^2_{cr})$ to $BMO$ and norm 1 from  $L^2(\ell^2)$ to $L^2(\hat G).$
By the interpolation result Lemma \ref{JM12}, we get 
\begin{corollary} \label{cor1}
 Assume $(h_k)$ is a $\psi$-lacunary sequence for some conditionally negative $\psi$. We have that, for any $p>2, x=\sum_kc_k\la _  {h_k} $, 
\begin{eqnarray}
 \|x\|_{L^p}\leq c^{\frac {p-2}{2p}}_\delta p\max\{\|(\sum_k |c_k|^2)^\frac12 \|_{L^ p},\|(\sum_k |c_k^*|^2)^\frac12 \|_{L^ p}\}. \label{psqrtp}
 \end{eqnarray}
 By duality, we get, for any $1<p<2$,
 \begin{eqnarray}
  \|(c_k)_k\|_{S^p(\ell_{cr}^2)}  \lesssim \inf\{\|x\|_p; \hat x(h_k)=c_k\}.\label{10}
 \end{eqnarray}
 \end{corollary}
 We will prove a column version of (\ref{10}) in the next section.
 \begin{remark} Corollary \ref{cor1} can also be obtained by combining the noncommutative $H^\infty$-calculus techniques developed in \cite{JMX06} (page 118) and the dilation theory proved in \cite{R08}. There is another approach via noncommutative Riesz transforms (\cite{JMPX}). The order of the equivalence  constants   in (\ref{psqrtp})  is better than what implied by these two approaches. It is unclear whether one can expect an optimal order $\sqrt p$ like the case of the integer group.
  \end{remark}
 
   \begin{remark}   If,  $G=\F_n,\psi $ is the reduced word length, it is also easy to verify that $\psi$-lacunary set is $B(2)$ in the sense of W. Rudin, so it is a $\Lambda(4)$ set by Harcharras's work\cite{Ha99}. This does not seem clear  for $B(p)$ with $p>2$.   
 \end{remark}
 \begin{remark}  The sequence of free generators $\{g_i,i\in {\Bbb N}\}$ of $\F_\infty$ is a $\psi$-lacunary sequence for some $\psi$. Indeed, let $\pi$ be the group homomorphism on $\F_\infty$ sending $g_i$ to $g_{i}^{2^i}$. Then $\psi(h)=|\pi(h)|$ is a conditionally negative function.
 \end{remark}

   \begin{remark}  One can extend (\ref{10})  to the range $0<p\leq1$ as a Khintchine-type inequality
    \begin{eqnarray}
 \|(c_k)_k\|_{S^p(\ell_{cr}^2)} \lesssim \|\sum_k c_k\la_{h_k} \|_{\frac p2}^2.\label{Khin}
 \end{eqnarray}
   applying Pisier-Ricard's theorem \cite{PR17}. For the $p=1$ case, one may follow Haagerup-Musat's argument in \cite{HM07} to get a better constant in (\ref{Khin}).
\end {remark}

  \begin{remark}  The conditionally negativity assumption of $\psi$ is needed merely by the interpolation result Lemma \ref{JM12}. The arguments for other results of this section need the assumptions (\ref{psi1}), (\ref{psi2}) only.
\end {remark}

\begin{remark}\label{finalremark} Let $P_t$  be the Poisson semigroup for bounded functions on the torus ${\Bbb T}$. As we pointed out before, the semigroup BMO associated with $P_t$  coincides with the classical BMO. Let $S_t=P_t\otimes id_{M_n}$ be its extension to the von Neumann algebra of  bounded  $n$ by $n$ matrix-valued functions  on ${\Bbb T}^d$.  The semigroup  BMO$_c$ associated with $S_t$  coincides with the matrix-valued BMO$_{so}$ introduced in the literature (e.g. \cite{NPTV02},\cite{Ps00}). Note for the dyadic BMO$_{so}$-norm, easy examples by Rademacher functions show  that $\|x\|_{BMO^d_{so}}\simeq ^{c_n }\|x^*\|_{BMO^d_{so}}$ with $c_n\simeq \sqrt n$ being optimal. There has not been an easy way to show that $\sqrt n$ is also optimal for the usual (non-dyadic) BMO$_{so}$ norm.\footnote{The BMO by dyadic BMO trick does not help on producing a concrete example.}   Theorem \ref{main1} provides such a way by taking $f=\sum_{0<k\leq n}  e_{1,k}z^{2^k}$.
   \end{remark}

	
\section{Proof of Theorem \ref{theorem2}}
	

Through out this section, we assume  $(G, \leq)$ is a countable (non-abelian) discrete group with a bi-invariant total order. 
This is equivalent to say that $G$ contains a normal subsemigroup  $G_+$
 such that, for $G_-=(G_+)^{-1}$,
$$G_+ \cup G_{-}=G, G_{+} \cap G_{-}=\{e\}.$$  
In this case, one has $G_+=\{g\in G; g\geq e\}$ and $x\leq y$ iff $x^{-1}y\in G_+.$ We use the notation $x<y$ if $x\leq y$ and $x\neq y$.

\subsection{Noncommutative $H^p$- space}


The analytic noncommutative Hardy spaces  make sense in the general context of Arveson's subdiagonal operator algebras (\cite{A98}). Its definition is quite simple in the ordered group case. 

Let $\mathcal{L}(G)$ be the group von Neumann algebra with the canonical trace $\tau_G$.
Let $(\mathcal{M},\tau)$ be
a semifinite von Neumann algebra.
Let $\mathcal{N}=\mathcal{L}(G)\overline{\otimes} \mathcal{M}$ with the trace $\tau_G\otimes \tau$. For $1\leq p\leq\infty$,  let $\mathcal{A}_p\subset L^p({\cal N})$ be the collection of all the finite sum $\sum c_g\otimes\lambda_g$ with $g\in G_+, c_g\in L^p({\cal M}).$
Let  $H^p(\mathcal{N})$ be the norm (resp. weak operator) closure of ${\cal A}_p$ in $L^p({\cal N})$ for $1\leq p<\infty$ (resp. $p=\infty$). The classical factorization, interpolation and duality results extend to the noncommutative Hardy spaces $H^p(\mathcal{N})$ (see \cite{PX03}, section 8). We will  need the following factorization theorem (Theorem 4.3 of  \cite{MW98a}, Theorem 3.2 of \cite{BO20}), 
\begin{lemma}\label{MW98}  
Given any $x\in H^1(\mathcal{N})$ and $\varepsilon>0$,
 there exist
$y, z\in H^2(\mathcal{N})$
such that $x=yz$ and $\|y\|_2\|z\|_2\leq\|x\|_1+\varepsilon$. 
\end{lemma}

For each $g\in G_+$. Put
$L_g=\{h: g\leq h\leq g^2\}$.
For $E\subset G_+$, let $N(E, g)$ be the number of elements of $E\cap L_g$, i.e.,
$N(E, g)=\#(L_g\cap E)$. Following Rudin's terminology (\cite{Ru62}), we say $E\subset G_+$ is lacunary, if  there is a constant $K$ such that
$$N(E)=\sup_{g\in G_+} N(E,g)\leq K.$$ 
For a general subset $E\subset G$, let $E_+=E\cap G_+, E_-=E-E_+$. We say $E$ is lacunary if $N(E)=N(E_+)+N((E_-)^{-1})$ is finite.

\begin{theorem}\label{theorem4.2}
Assume  that $E$ is a lacunary subset of $G_+$. Then,
for any sequence $c_k\in L^1({\cal M})$, and
any sequence $\{g_k\}_{k=1}^\infty\in E$, we have
 \begin{eqnarray}
 \ \ \ \ \ \  \  \|(c_k)\|_{L^1({\cal M},\ell^2_{cr})}\leq 2K^\frac12\inf \{\tau_G\otimes\tau |x|; \hat x(g_k)=c_k, \supp\hat x\subset G_+\},
 \end{eqnarray}
 for any   $x\in L^1({\cal N})$.
\end{theorem}
\begin{proof} 
By approximation, without loss of generality, we assume $E$ is a finite set and 
$\mathcal{N}$ is a finite von Neumann algebra.
For $x\in H^1(\mathcal{N})$ and $\varepsilon>0$,
 by Lemma  \ref{MW98},   there exist
$y, z\in H^2(\mathcal{N})$
such that $x=yz$ and $\|y\|_2\|z\|_2\leq\|x\|_1+\varepsilon$. Let us use the notation  $$\hat x(g)=\tau_G(x\la_g^*).$$
Given an element $g_i \in E$ with
$\widehat{x}(g_i)\neq0$, we have 
\begin{equation}\label{convolution of fourier tranform2}
\widehat{x}(g_i)= \sum_{e\leq h\leq g_i}\widehat{y}(h)\widehat{z}(h^{-1}g_i),
\end{equation}
since $\widehat{y}(g)=\widehat{z}(g)=0, ~\forall~g<e$.

Thus the sum in (\ref{convolution of fourier tranform2}) can be split in two parts;
\begin{equation}\label{sum split in two}
\widehat{x}(g_i)=\sum_{e\leq h\leq g_i < h^2}\widehat{y}(h)\widehat{z}(h^{-1}g_i)+
\sum_{e\leq h\leq h^{2}\leq g_i}\widehat{y}(h)\widehat{z}(h^{-1}g_i).
\end{equation}
Let
$$Z_i=\sum_{e\leq h\leq h^{2}< g_i}\lambda_{h} \otimes\widehat{z}(h)
=\sum_{e\leq h\leq  g_i< h^2}\lambda_{h^{-1}g_i}\otimes\widehat{z}(h^{-1}g_i).$$ 
Similarly let $$Y_i=\sum_{e\leq h\leq h^{2}\leq g_i}\lambda_{h}\otimes\widehat{y}(h).$$
It follows that
\begin{eqnarray*}
Z_i(\lambda_{g_i^{-1}}\otimes1)
&=&\sum_{e\leq h\leq g_i < h^2}\lambda_{h^{-1}}\otimes\widehat{z}(h^{-1}g_i)
\end{eqnarray*}
and
\[
(\lambda_{g_i^{-1}}\otimes1)Y_i
=\sum_{e\leq h\leq h^{2}\leq g_i }\lambda_{g_i^{-1}h}\otimes\widehat{y}(h).
\]
Let \begin{eqnarray}
A_i:&=&(\tau_G\otimes 1)(yZ_i(\lambda_{g_i^{-1}}\otimes1))\label{Ai}\\ 
B_i:&=&(\tau_G\otimes 1)((\lambda_{g_i^{-1}}\otimes1)Y_iz).\label{Bi}
\end{eqnarray}
Then 
\begin{eqnarray}
\hat x(g_i)=A_i+B_i,
\end{eqnarray}
because
\begin{eqnarray*}
A_i=&&(\tau_G\otimes 1)(yZ_i(\lambda_{g_i^{-1}}\otimes1))\\
&=&
(\tau_G\otimes 1) ((\sum_{g\geq e}\lambda_g\otimes\widehat{y}(g))
(\sum_{e\leq h\leq g_i< h^2}\lambda_{h^{-1}}\otimes\widehat{z}(h^{-1}g_i)))\\
&=&\tau_G\otimes 1 ((\sum_{g\geq e}\sum_{e\leq h\leq
g_i< h^2}\lambda_{gh^{-1}}\otimes\widehat{y}(g)\widehat{z}(h^{-1}g_i))\\
&=&\sum_{e\leq h\leq g_i< h^2} \widehat{y}(h)\widehat{z}(h^{-1}g_i).
\end{eqnarray*}
and
\begin{eqnarray*}
B_i
&=&\tau_G\otimes 1 ((\lambda_{g_i^{-1}}\otimes1)Y_iz)\\
&=&(\tau_G\otimes 1)((\sum_{e\leq h\leq h^{2}\leq g_i  }\lambda_{g_i^{-1}h}\otimes\widehat{y}(h))
(\sum_{g\geq e}\lambda_g\otimes\widehat{z}(g)))\\
&=&(\tau_G\otimes 1)((\sum_{e\leq h\leq h^{2}\leq g_i}\lambda_{g_i^{-1}h}\otimes\widehat{y}(h))
(\sum_{f\leq g_i}\lambda_{f^{-1}g_i}\otimes\widehat{z}(f^{-1}g_i)))\\
&=&(\tau_G\otimes 1)(\sum_{e\leq h\leq h^{2}\leq g_i }\sum_{f\leq g_i}
\lambda_{g_i^{-1}h}\lambda_{f^{-1}g_i}\otimes\widehat{y}(h)\widehat{z}(f^{-1}g_i))\\
&=&\sum_{e\leq h\leq h^{2}\leq g_i } \widehat{y}(h)\widehat{z}(h^{-1}g_i)
\end{eqnarray*}

Applying the convexity of $\tau_G$ and Jensen's inequality to (\ref{Ai}), we have
\begin{eqnarray*}
\|(A_i)_1^n\|_{L^1(\mathcal{M}, \ell_r^2)}
&\leq&\tau_G\otimes \tau[ (\sum_{i}|(yZ_i\lambda_{g_i^{-1}}\otimes1)^*|^2)^{\frac{1}{2}}]\\
&=&\tau_G\otimes \tau[ (y(\sum_{i}Z_iZ_i )y^*)^{\frac{1}{2}}]\\
&\leq&(\tau_G\otimes \tau(|y|^2))^{\frac{1}{2}}(\tau_G\otimes
\tau\sum_{i}Z_iZ_i^* )^{\frac{1}{2}}\\
&=&(\tau_G\otimes \tau(|y|^2))^{\frac{1}{2}}(\sum_{i}\sum_{e\leq h\leq
g_i< h^2}\|\widehat{z}(h)\|_{L^2(\mathcal{M})}^2)^{\frac{1}{2}}.
\end{eqnarray*}
On the other hand,
we note that $e\leq h\leq g_i<h^2$ implies that
$g_i\in L_{h}$.
Since $N(E, g)\leq K$ we get
\begin{eqnarray*}
\|(A_i)_1^n\|_{L^1(\mathcal{M}, \ell_r^2)}&
\leq&(\tau_G\otimes \tau(|y|^2))^{\frac{1}{2}}(K\sum_{h}\|\widehat{z}(h)\|_{L^2(\mathcal{M})}^2)^{\frac{1}{2}}\\
&= &K^\frac12 \|z\|_{L^2(\mathcal{N})}\|y\|_{L^2(\mathcal{N})}\leq K^\frac12(\|x\|_{L^1(\mathcal{N})}+\varepsilon).
\end{eqnarray*}
 We now take care of $(B_i)_i$. Similarly,
\begin{eqnarray*}
\|(B_i)_1^n\|_{L^1(\mathcal{M}, \ell_c^2)}&\leq& (\tau_G\otimes tr)[z^*(\sum_{i}|\lambda_{g_i^{-1}}\otimes1Y_i|^2)z]^{\frac{1}{2}}\\
&\leq&(\tau_G\otimes \tau(|z|^2))^{\frac{1}{2}}(\tau_G\otimes tr(\sum_{i}|\lambda_{g_i^{-1}}\otimes1Y_i|^2))^{\frac{1}{2}}\\
&=&\|z\|_2(\sum_{i}\tau_G\otimes
\tau(Y^*_iY_i))^{\frac{1}{2}}\\
&=&\|z\|_2(\sum_{i}\sum_{e\leq h\leq  h^{2}\leq g_i}\|\widehat{y}(h)\|_{L^2(\mathcal{M})}^2)^{\frac{1}{2}}.
\end{eqnarray*}
Note the condition $e\leq h\leq   h^2\leq g_i$ implies $$h^{-1}g_i\leq g_i\leq hg_i<g_ih^{-1}g_i$$ and 
$$ h' \leq g_i\leq h'^2$$ with $h'= {h^{-1} g_i}\geq e$ because ``$\leq$" is bi-invariant. We then get
$$\sum_{e\leq h\leq  h^{2}\leq g_i}\|\widehat{y}(h)\|_{L^2(\mathcal{M})}^2\leq \sum_{e\leq h' \leq g_i\leq h'^2}\|\widehat{y}(h')\|_{L^2(\mathcal{M})}^2$$
By the lacunary assumption $N(E)\leq K$, we get  
\begin{eqnarray*}
\|(B_i)_1^n\|_{L^1(\mathcal{M}, \ell_c^2)}
&\leq&\|z\|_2(\sum_{i}\sum_{e\leq h'\leq g_i\leq  h'^{2}}\|\widehat{y}(h')\|_{L^2(\mathcal{M})}^2)^{\frac{1}{2}}\\
&\leq& K^\frac12 \|z\|_{L^2(\mathcal{N})}\|y\|_{L^2(\mathcal{N})}\leq K^\frac12(\|x\|_{L^1(\mathcal{N})}+\varepsilon).
\end{eqnarray*}
Therefore,
\begin{eqnarray*}
\|(\widehat{x}(g_i))_1^n\|_{L^1(\mathcal{M}, \ell_{rc}^2)}
&\leq& \|(B_i)_1^n\|_{L^1(\mathcal{M}, \ell_c^2)}+\|(A_i)_1^n\|_{L^1(\mathcal{M}, \ell_r^2)}\\
&\leq& 
 2K^\frac12(\|x\|_{L^1(\mathcal{N})}+\varepsilon),
\end{eqnarray*}
This completes the proof by letting $\varepsilon\rightarrow 0$.
\end{proof}

For a finite sum $x=\sum_g \hat x(g)\otimes\la_g $ with $\hat x(g)\in L^1({\cal M})$, we set $$x=\sum_{g\geq e} \hat x(g)\otimes\la_g,\ \ x_-=\sum_{g<e} \hat x(g)\otimes\la_g .$$
and define $\|x\|_{ReH_1}=\tau_G\otimes \tau(|x_+|+|x_-|).$
Let Re$H_1({\cal N})$ be the  norm closure of all finite sums $\sum_g \hat x(g)\la_g $. 
\begin{corollary}\label{cor4.3}
Assume  that $E$ is a lacunary subset of $G$. Then,
for any sequence $c_k\in L^1({\cal M})$, and
any sequence $\{g_k\}_{k=1}^\infty\in E$, we have
 \begin{eqnarray}
  \|(c_k)\|_{L^1({\cal M},\ell^2_{cr})}\simeq  \inf \{\|x\|_{ReH_1}; \hat x(g_k)=c_k, \},
 \end{eqnarray}
 for any   $x\in ReH^1({\cal N})$. Moreover, the projection $P$ from $L^2({\cal N})$ onto the subspace expanded by $\la(E)$ extends to a completely bounded map on Re$H^1({\cal N})$.
\end{corollary}
It is proved in \cite{MW98b} and \cite{BO20} that the dual of Re$H_1$ can be identified  as a BMO space and the complex interpolation between this BMO space and Re$H_1({\cal N})$ is $L^p({\cal N})$ for $1<p<\infty$. We then get
\begin{corollary}\label{cor4.4}
For any  sequence $\{g_i\}_{i=1}^n$ in a lacunary subset $E\in G$
\begin{eqnarray}
\|\sum_{i=1}^n\lambda_{g_i}\otimes c_{g_i}\|_p
\approx\|(c_{g_i})_1^n\|_{L^p(\mathcal{M}, \ell_{rc}^2)}, 0<p<\infty.\label{1pinfty}
\end{eqnarray}
Moreover, the projection $P$ from $L^2({\cal N})$ onto the subspace expanded by $\la(E)$ extends to a completely bounded map on $L^p({\cal N})$ for all $1<p<\infty$.
\end{corollary}
\begin{proof} The $1<p<\infty$ case follows from the aforementioned duality and interpolation results proved in \cite{MW98b}, \cite{BO20}. The $0<p<1$ case follows from \cite[Corollary 2.2]{PR17} and \cite[Theorem 2.6]{C19}.
\end{proof}
\begin{remark}
 It would be interesting to find whether Pisier-Ricard's argument \cite{PR17} can push Theorem \ref{theorem4.2} to the $p<1$ case.
\end{remark}

\section {The case of free groups}

  Let $G={\Bbb F}_2$ be the nonabelian group with  two free generators $a,b$.  Denotes by $|g|$ the reduced word length of $g\in\F_2$. Every $g \in {\Bbb F}_2$ can be uniquely written as  \begin{eqnarray}\label{g}
g=a^{j_1}b^{k_1}\cdots a^{j_{N}}b^{k_{N}}
 \end{eqnarray} with $j_i,k_i\in {\Bbb Z}$ and $j_i \neq0$ for $1<i\leq N$ and $k_i\neq0$ for $1\leq i< N.$
 Set  the $q$-length of $g$ to be
\begin{eqnarray}\|g\|_q=\sum_i |j_i| ^q+\sum_{i}|k_i|^q.\label{qlength}
\end{eqnarray}  Then $\psi:g\rightarrow \|g\|_q$ is a conditionally negative function for all $0<q\leq 2$. When $q=1$, $\|g\|_q$ is the reduced word length. Its conditional negativity was studied in \cite{Ha79}. All results contained in Section 2, 3 apply to these $\psi$. In particular, all $\|\cdot\|_q$-lacunary sequences are  completely unconditional in $L^p(\hat \F_2)$ for all $0<p<\infty$. This is not clear for $p=\infty$.

  We say a subset $A\in G$ is a completely (unconditional) Sidon set, if $\{\lambda_h, h\in A\}$ is  (completely) unconditional in ${\mathcal L}(\hat G)$, i.e. there exists a constant $C_A$ such that
$$\|\sum_{h_k\in A} \eps_k c_k\lambda_{h_k}\|\leq C_A\|\sum_{h_k\in A}  c_k\lambda_{h_k}\|,$$
for any  choice $\eps_k=\pm$, $c_k\in {\Bbb C}$ (resp.  $K(H)$).
Given a conditionally negative $\psi$ with $ker\psi=\{e\}$, we say a subset $A\in G$ is a (complete)  $\psi$ Paley-set, if there exists a constant $C_A$ such that
\begin{eqnarray}\label{lainfty}
\|\sum_{h_k\in A}  c_k\lambda_{h_k}\|_{BMO(\psi)}\leq C_A\max\{\|\sum_{h_k\in A}  |c_k|^2\|^\frac12,\|\sum_{h_k\in A}  |c_k^*|^2\|^\frac12\},\end{eqnarray}
for any  choice of finite many $c_k\in {\Bbb C}$ (resp.  $K(H)$). These definitions coincide with the classical ``Sidon" and ``Paley" set, when $G={\Bbb Z}$, and $\psi$ is the word length on ${\Bbb Z}$.  In that case, every Paley set is a Sidon set.  One may wonder to what extent this is still true. In the case that $G={\F}_2$ and $\psi$ being the reduced word length (or $q$-length defined in (\ref{qlength})), every length-lacunary set is a Paley set and a completely $\Lambda(p)$ set for all $2<p<\infty$ as showed in this article, the question is 
\medskip

{\bf Question}: Suppose $h_k\in \F_2$ is a length-lacunary sequence, e.g. $\frac {|h_{k+1}|}{h_k}>2$. Is $\{g_k\}$  a (completely) unconditional Sidon set? i.e. does there exists a constant $C_{\delta}$ such that
$$\|\sum_{k} \eps_k c_k\lambda_{h_k}\|\leq C_{\delta}\|\sum_{k}  c_k\lambda_{h_k}\|,$$
for any choices $\eps_k=\pm1$ and $c_k\in {\Bbb C}$ (resp. $K(H)$)?
 
 The transference method used in the work \cite{CXY13} is quite powerful  for the study of harmonic analysis on the quantum tori. A similar method applies to the free group case.  For $g\in\F_2$ in the form of (\ref{g}), let 
 $$|g|_z=|\sum_{i=1}^N j_i|^2+|\sum_{i=1}^N k_i|^2.$$ 
Then \begin{eqnarray}\label{psiz}
\psi_z:g\mapsto |g|_z
\end{eqnarray}
 is another conditionally negative function on $\F_2,$ and the unbounded linear operator $L_z:\la_g\mapsto \psi_z\la_g$ generates a symmetric Markov semigroup on the free group von Neumann algebra ${\cal L}(\F_2)$.  For $(z_1,z_2) \in {\Bbb T}^2$, let $\pi_z$ be the *-homomorphism on   ${\cal L}(\F_2)$ such that 
$$\pi_z(\la_a)=z_1\la_a, \pi_z(\la_b)=z_2\la_b.$$
Given $x\in {\cal L}(\F_2)$, viewing $\pi_z(x)$ as an operator valued function on ${\Bbb T}^2$, one can see that 
\begin{eqnarray}
\pi_z^{-1}(\Delta\otimes id)\pi_z(x) =L_z(x),\label{nabla}
\end{eqnarray}
with $ \Delta$ the Laplacian on ${\Bbb T}^2$.
This identity allows one to  transfer  classical results  to free groups with   $L_z$ taking the role of the Laplacian, including the corresponding Paley's inequality proved in this article. The disadvantage is that  this transference method cannot     produce any helpful information on the large subgroup $ker\psi_z$.  At below, we will show that the second part of this paper implies  a Paley's theory on   $ker\psi_z$.  

 Let us first recall a bi-invariant order on free groups  $\mathbb{F}_2$.
 For notational convenience, we denote   the free generators by $x_1, x_2$. We define the ring
$\Lambda= \mathbb{Z}[ A, B]$
to be the ring of formal power series in the non-commuting variables $A$ and $B$.
Let $\mu$ be the group homomorphism from $\mathbb{F}_2$ to the group generated by $\{1+A,1+B\}$ in $\Lambda$
such that:
\begin{eqnarray*}
&\mu(a)=1+ A, \ \mu(a^{-1})=1-A+A^2-A^3+\cdots
\\
&\mu(b)=1+B,\  \mu(b^{-1})=1-B+B^2-B^3+\cdots
\end{eqnarray*}
Then $\mu$ is injective. Denote by ``$\leq$" the dictionary order on $\Lambda$ assuming  $0\leq B\leq A$.  To be precise:  Write the element of $\Lambda$
in a standard form, with lower degree terms preceding higher degree terms,
and within a given degree list the terms in sequence according to   the
dictionary ordering assuming $0\leq B\leq A$.
Compare two elements of $\Lambda$ by writing them both in standard form and ordering
them according to the natural ordering of the  coefficients at the first term at
which they differ.
We then formally define the ordering on the free group $\mathbb{F}_2$ by setting
$$g\leq h~\mbox{in}~\mathbb{F}_2{\ \ \rm if\ \ } \mu(g)\leq\mu(h)~\mbox{in}~\Lambda.$$

Let $J_A(g)$ (resp. $J_B(g)$) be the coefficient of the $A$ term (resp. B term) in $\mu(g)$; and $J_{AB}(g)$ (resp. $J_{BA}(g))$ be the coefficient of  the $AB$ term (resp. $BA$ term)  in $\mu(g)$. More general, for any word  X of $A,B$, denote by $J_X(g)$ the coefficient of the $X$ term  in $\mu(g)$. 
Note that   $J_A(g)=J_A(a^{J_A(g)})$, and   $J_B(g)=J_B(b^{J_B(g)})$. For $g\in \F_2$ in the form of (\ref{g}), that is \begin{eqnarray}
g=a^{j_1}b^{k_1}\cdots a^{j_{N}}b^{k_{N}}
 \end{eqnarray} with $j_i,k_i\in {\Bbb Z}$ and $j_i \neq0$ for $1<i\leq N$ and $k_i\neq0$ for $1\leq i< N,$ we get by direct computations, 
 \begin{eqnarray}
 J_A(g)&=& \sum_{s=1}^N j_s,\ \  \ \ J_B(g)=\sum_{s=1}^N k_s,\label{dJA} \\
 J_{AB}(g)&=&\sum_{1\leq s\leq t \leq N} j_s k_t,  \ J_{BA}(g)=\sum_{1\leq t<s \leq N} j_s k_t. \label{dJAB}
 \end{eqnarray}
From (\ref{dJA}), (\ref{dJAB}), we see that $$J_{AB}(g)+J_{BA}(g)=J_A(g)J_B(g).$$ Using that $\mu$ is a group homomorphism, we have \begin{eqnarray}
J_A(gh)=J_A(g)+J_A(h)\label{JA}\\
J_{AB}(gh)=J_A(g)J_B(h)+J_{AB} (g)+J_{AB}(h)\label{JAB},
\end{eqnarray} 
 
Let 
\begin{eqnarray*}
\F_2^0&=&ker\psi_z=\{g\in \F_2; J_A(g)=J_B(g)=0\},\\
\F_2^{00}&=&\{g\in \F^0_2;  J_{AB}(g)=0\}=\{g\in \F^0_2;  J_{AB}(g)=J_{BA}(g)=0\}.
\end{eqnarray*} 
Then $\F_2^0, \F_2^{00}$ are   subgroups because of (\ref{JA}), (\ref{JAB}), and $\F_2^0=ker \psi_z$ with $\psi_z$ defined in (\ref{psiz}).  For $g\in \F_2^0$, $g>e$ if $J_{AB}(g)>0$ since $J_{AA}(g)=0$.
 Recall we say a sequence of $\ell_n\neq0\in{\Bbb Z}$ is lacunary if   there exists a $\delta>1$ such that $\inf_n\frac{\ell_{n+1}}{\ell_{n}}\geq \delta.$
We then get the following property by definition.
\begin{prop}\label{lastprop}
Given a sequence $g_n\in \F_2$, then  $E=\{g_n, n\in{\Bbb N}\}$  is a lacunary subset of $\mathbb{F}_2$ if any of the following holds
\begin{itemize}
\item The sequence $J_A(g_n)\in {\Bbb Z}$ is lacunary.
 \item  $J_A(g_n)=0$ for all $n$ and  the sequence $J_B(g_n)  \in {\Bbb Z}$ is lacunary.
 \item $J_A(g_n)=J_B(g_n)=0$ for all $n$, and  $J_{AB}(g_n)$ is lacunary. 
 \end{itemize}
\end{prop}
\noindent For instance,   $\{x_1^{2^i}x_2^{k_i}\in \mathbb{F}_2: i,  k_i\in \mathbb{N}_+\}$  
and  $\{x_1^{2^k}x_2^{2^k}x_1^{-2^k} x_2^{-2^k};  k\in {\Bbb N}\}$ are  lacunary subset of $\F_2$. 
   \begin{remark}Corollary \ref{cor4.4} implies that the  sets $E$ given in Proposition \ref{lastprop} are all completely $\Lambda(p)$ sets (\cite{Ha99}). 
\end{remark}

\begin{corollary}\label{corlast}Suppose $(g_k)_k\in \F^0_2$ is a sequence with $(J_{AB}(g_k))_k\in {\Bbb Z}$ lacunary.
Then for any $(c_k)_{k}\in S^p(H)$, we have 
\begin{eqnarray}
\|(c_k) \|^p_{S^p(\ell_{rc}^2)}\simeq  tr\otimes \tau|\sum_kc_k\otimes \la_{g_k} |^p\label{lastp}
\end{eqnarray}
for all $0<p<\infty$. Moreover, for $p=1$, we have  
\begin{eqnarray}
&&\|(c_k) \|_{S^1(\ell_{rc}^2)}\nonumber\\
&\simeq& \inf\{ tr\otimes\tau\big(| \sum_{J_{AB}(g)\geq0} \hat x(g)\otimes\la_{g}| +| \sum_{J_{AB}(g)<0} \hat x(g)\otimes\la_{g}|\big) \}\label{last1}
\end{eqnarray}
Here the infimum runs over all  $x\in L^1(\hat \F_2)\otimes S^1(H) $ with $\hat x(g_k)=c_k$.
\end{corollary}
\begin{proof} (\ref{lastp}) follows from Corollary \ref{cor4.4}.
For (\ref{lastp}),   we only need to prove the relation ``$\lesssim$", the other direction is trivial.   Since $\F_2^{0}$ and $\F_2^{00}$ are   subgroups, the projection $P_0$ (and $P_{00}$) onto $L^1(\hat \F_2^0)$ (and $L^1(\hat \F_2^0)$)  is completely contractive. 
Given   $x\in L^1(\hat \F_2)\otimes S^1(H)$ with $\hat x(g_k)=c_k$, let 
$y=P_{0}x-P_{00}x$ then we still have  $\hat y(g_k)=c_k$. 
By Corollary \ref{cor4.3}, we have
\begin{eqnarray*}
&&\|(c_k)_{1\leq k\leq n} \|_{S^1(\ell_{rc}^2)}\\
&\lesssim&   tr\otimes\tau| \sum_{g\geq e} \hat y(g)\otimes\la_{g}| +tr\otimes\tau| \sum_{g<e} \hat y(g)\otimes\la_{g}|\\
  &=& tr\otimes\tau| \sum_{g\in \F_2^0, J_{AB}(g)>0} \hat x(g)\otimes\la_{g}| +tr\otimes\tau| \sum_{g\in\F_2^0,J_{AB}(g)<0} \hat x(g)\otimes\la_{g}|\\
  &\leq& tr\otimes\tau\big(| \sum_{g\in \F_2^0, J_{AB}(g)\geq 0} \hat x(g)\otimes\la_{g}| +|P_{00}x|+| \sum_{g\in\F_2^0,J_{AB}(g)<0} \hat x(g)\otimes\la_{g}|\big)\\
  &\leq& 2tr\otimes\tau(|  \sum_{g\in \F_2^0, J_{AB}(g)\geq0} \hat x(g)\otimes\la_{g}| +| \sum_{g\in\F_2^0,J_{AB}(g)<0} \hat x(g)\otimes\la_{g}|)\\
  &\leq& 2tr\otimes\tau(|  \sum_{g\in \F_2, J_{AB}(g)\geq0} \hat x(g)\otimes\la_{g}| +| \sum_{g\in\F_2,J_{AB}(g)<0} \hat x(g)\otimes\la_{g}|).
\end{eqnarray*} 
\end{proof}
 
\begin{remark} The associated positive semigroup of any total  order  (including the one introduced above) on free groups is NOT represented by a regular language (\cite{HS17}). This increases the mystery of the associated   noncommutative Hardy spaces (norms).  Corollary \ref{corlast}  shows that  there are more transparent alternatives (e.g. (\ref{last1}) ) of the noncommutative real $H^1$-norm that may be used to formulate   the corresponding Paley's inequalities.
\end{remark}
\begin{remark} 
  Interested readers are invited  to prove  a similar theory by computing  $J_{AAB}(g)$.
\end{remark}

\section*{ Acknowledgments}

\noindent Mei is partially supported by NSF grant DMS 1700171. Han is partially supported by NSFC grant.
\bibliographystyle{amsplain}




\end{document}